\documentclass[12pt]{amsart}
\usepackage{amsmath,amssymb,amsfonts,amsthm,amsopn,dsfont}
\usepackage[dvips]{graphicx}
\usepackage{srcltx}
\usepackage{setspace}
\newtheorem{theorem}{Theorem}[section]

\numberwithin{equation}{section}

\newtheorem{proposition}[theorem]{Proposition}
\newtheorem{remark}[theorem]{Remark}

\renewcommand{\ell}{l}
\renewcommand{\epsilon}{\varepsilon}

\def\ga{\gamma}

\def\px{\langle x \rangle}

\def\rd{\bR^d}

\newcommand{\Fur}{\mathcal{F}}

\def\R{\right)}

\def\<{\left<}
\def\>{\right>}

\def\mv1{M_v^1}

\def\mn{(m,n)}
\def\mn'{(m',n')}

\newcommand{\bC}{\mathbb{C}}
\def\N{\mathbb{N}}
\def\R{\mathbb{R}}
\def\C{\mathbb{C}}
\def\rd{\mathbb{R}^d}

\begin{document}

\title[]{Pointwise decay and smoothness for semilinear elliptic equations and travelling waves}
\begin{abstract}
We derive sharp decay estimates and prove holomorphic extensions for the solutions of a class of semilinear nonlocal elliptic equations
with linear part given by a sum of Fourier multipliers with finitely smooth symbols at the origin. Applications concern the decay and 
smoothness of travelling waves for nonlinear evolution equations in fluid dynamics and plasma physics.
\end{abstract}
\author{Marco Cappiello \and Fabio Nicola}
\address{Dipartimento di Matematica,  Universit\`{a} degli Studi di Torino,
Via Carlo Alberto 10, 10123
Torino, Italy
} \email{marco.cappiello@unito.it}
\address{Dipartimento di Scienze Matematiche, Politecnico di
Torino, Corso Duca degli
Abruzzi 24, 10129 Torino,
Italy}
\email{fabio.nicola@polito.it}
\subjclass[2000]{}
\date{}
\maketitle

\section{Introduction}
The aim of this note is to derive pointwise decay estimates for a class of semilinear elliptic equations in $\R^d$ of the form
\begin{equation} \label{equation}
p(D)u=F(u),
\end{equation}
where $F:\bC\to\bC $ is any measurable function satisfying 
\begin{equation}\label{nonlin}
|F(u)|\leq C_K |u|^{p}
\end{equation}
for some $p>1$, uniformly for $u$ in compact subsets $K\subset\bC$, and 
$p(D)$ is a Fourier multiplier with symbol of the form 
\begin{equation}\label{simbolo}
p(\xi)=\sum_{j=0}^h p_{m_j}(\xi)
\end{equation}
where $p_{m_j}$ are positively homogeneous functions of degree $m_j$, with $0=m_0<m_1<m_2<\ldots<m_h=M$.  
\par
Equations of this form frequently appear in Mathematical Physics, in particular in the theory of travelling wave solutions of nonlinear evolution
equations in fluid dynamics (KdV-type, long-wave-type, Benjamin-Ono and many others) or in the frame of 
stationary Schr\"odinger equations. In \cite{BL1, BL2}, Bona and Li proved that when the symbol $p(\xi)$ is smooth and elliptic on $\R$,
then every solution $u$ of \eqref{equation} which tends to $0$ at infinity indeed exhibits an exponential decay of the form $e^{-c|x|}$ 
for some positive constant $c$. In addition, if $p(\xi)$ satisfies uniform analytic estimates on $\R$, then also $u$ is analytic and it extends to a holomorphic function in a strip of the form $\{x+iy \in \C : |y| < T\}$ for some $T>0$. Similar results have been proved later for more general pseudodifferential equations in arbitrary dimension, see \cite{A3, CGR2, CGR3, CN1, CN2}. More recently in \cite{CGRNag} the first author et al. considered the case when the functions $p_{m_j}(\xi)$ in \eqref{simbolo} may present some \emph{finite} smoothness at $\xi=0$. As a difference with respect to the case of $C^\infty$ symbols, a finite smoothness of the symbol may determine a loss of the exponential decay observed in \cite{A3, BL1, BL2, CGR2, CGR3, CN1, CN2} and the appearance of an algebraic decay. This phenomenon has been observed in \cite{AT, Benjamin, M, Ono} on the travelling waves of the Benjamin-Ono equation and its generalizations, which solve equations of the form \eqref{equation}. Recently, in \cite{CGRNag}, $L^2$-algebraic decay estimates have been proved for the homoclinic solutions of \eqref{equation}. In \cite{CNDCDS} we improved the latter result by proving \textit{pointwise} decay estimates. Namely, we proved that if 
\begin{equation}\label{ellitticita2}
|p(\xi)|\geq C\langle \xi\rangle^M,\quad \xi\in\rd,\ \xi\not=0,
\end{equation}
for some $C>0$ (as usual $\langle \xi\rangle=(1+|\xi|^2)^{1/2}$), then
every solution $u$ of \eqref{equation} such that $\px^{\varepsilon_o}u \in L^\infty (\R^d)$ for some $\epsilon_0>0$ actually satisfies 
\begin{equation}\label{decDCDS}
\px^{d+m}u \in L^\infty (\R^d),
\end{equation}
where $m$ is the smallest $m_j$ for which $p_{m_j}$ is not smooth. This result however has been proved under the additional condition $m>0$, namely when the zero order term $p_0(\xi)$ is smooth in $\rd$ and hence constant: $p_0(\xi)=p_0\in\mathbb{C}$ (and $p_0\not=0$ by \eqref{ellitticita2}). The decay estimate \eqref{decDCDS} is sharp as can be checked on some basic examples, cf. \cite{CGRNag, CNDCDS}. \par
 It is natural to look for similar results in the presence of a non constant zero other term $p_0(\xi)$, which corresponds to $m=m_0=0$ in the above discussion (by letting $\xi\to 0$ in \eqref{ellitticita2} we see that $p_0(\xi)\not=0$ for $\xi\not=0$). This situation occurs for instance in the study of travelling wave type solutions to nonlinear equations from fluid dynamics and plasma physics. To be definite let us consider the following equation 
\begin{equation}\label{LCeq}
v_t + \mu H v_x + \beta H v_{xx}- \nu v_{xx} +2vv_x=0, \qquad t \geq 0 , x \in \R,
\end{equation}
where $\beta, \mu, \nu$ are real parameters and $H=H(D)$ stands for the Hilbert transform, namely the Fourier multiplier with symbol $h(\xi)=-i\,\textrm{sign}\,\xi$. The equation \eqref{LCeq} has been treated in \cite{LC} in view of its connections with the modelization of plasma turbulence.
Observe moreover that when $\beta=1$ and $\mu=\nu=0,$ the equation \eqref{LCeq} reduces to the Benjamin-Ono equation (cf.\ \cite{Benjamin, Ono}), whereas for $\beta =0$ and $\mu =-1$, we obtain the Sivashinsky equation
\begin{equation}\label{Siveq}
v_t - H v_x - \nu v_{xx} +2vv_x=0, \qquad t \geq 0 , x \in \R,
\end{equation}
which governs the dynamics of wrinkled flame fronts, cf.\ \cite{Siv1, TFH}.
When looking for solutions of \eqref{LCeq} of the form $v(t,x)=u(x-ct), c \in \R$, we are reduced to consider the equation 
\begin{equation}
\label{travLCeq}
-cu'+\mu H u'+\beta Hu''-\nu u'' +2uu'=0,
\end{equation}
which in turn can be written in the conservative form 
\begin{equation} \label{cons}
-cu+\mu Hu+\beta Hu'-\nu u' =-u^2.
\end{equation}
Notice that \eqref{cons} is of the form \eqref{equation}, \eqref{simbolo} with $F(u)=-u^2$ and 
\begin{equation} \label{simbcons}
p(\xi)= -c-i\mu\, \textrm{sign}\, \xi + \beta |\xi| -i\nu \xi.\end{equation}
By direct computation one can verify that if $\mu <0, \nu >0$ and $c= \beta \mu/ \nu$, the equation \eqref{cons} admits the solution
\begin{equation}\label{solLC} u(x)= -\frac{2\nu x +2b \beta}{x^2+b^2},
\end{equation}
with $b = -\nu/\mu$. This solution decays like $O(|x|^{-1})$ for $|x| \to \infty$ and admits a holomorphic extension in the strip $\{x+iy \in \C: |y|< b\}.$
Notice that the symbol $p(\xi)$ in \eqref{simbcons} is given by
$$p(\xi) = \beta  \left( -\frac{\mu}{\nu}+|\xi| \right)-i(\nu \xi + \mu\, \textrm{sign}\,\xi),$$
which satisfies the condition \eqref{ellitticita2} if, in addition to $\mu<0,\,\nu>0$, we have $\beta\not=0$.

\par
Motivated by the example above, in this note we prove pointwise decay estimates and holomorphic extensions for the solutions of the general equation \eqref{equation}, \eqref{nonlin}, \eqref{simbolo}, \eqref{ellitticita2}, extending the results proved in \cite{CNDCDS} to the case $m=m_0=0$. \par Namely, we have the following results.

\begin{theorem}\label{mainteo}
With the above notation, assume \eqref{nonlin}, \eqref{simbolo}, \eqref{ellitticita2}. Let $u$ be a distribution solution of $p(D)u=F(u)$, satisfying $\langle x\rangle ^{\epsilon_0} u\in L^\infty(\R^d)$ for some $\epsilon_0>0$. Then $\langle x\rangle ^{d} u\in L^\infty (\R^d)$. 
\end{theorem}

We have also the following estimates for the derivatives, when $F$ is smooth.
\begin{theorem}\label{mainteo2}
Assume \eqref{simbolo}, \eqref{ellitticita2} and let $F\in C^\infty(\mathbb{C})$, with $F(0)=0$, $F'(0)=0$.
Let $u$ be a distribution solution of $p(D)u=F(u)$, satisfying $\langle x\rangle ^{\epsilon_0} u\in L^\infty(\rd)$ for some $\epsilon_0>0$.\par
 Then $u$ is smooth and satisfies the estimates
\begin{equation}\label{stime2}
\|\langle \cdot \rangle^{d+|\alpha|-\varepsilon}\partial^\alpha u\|_{L^\infty}<\infty,\quad \alpha\in\mathbb{N}^d,
\end{equation}
for every $\varepsilon >0$.\par
In dimension $d=1$, we have the slightly stronger estimate
\begin{equation}\label{stime2bisbis}
\|\langle \cdot \rangle^{1+\alpha}\partial^\alpha u\|_{L^\infty}<\infty,\quad \alpha\in\mathbb{N}.
\end{equation}
\end{theorem}
Here it is meant that $F$ is smooth with respect to the structure of $\C$ as a real vector space, and with $F'$ we denote its differential. Observe, in particular, that the above assumptions on $F$ imply \eqref{nonlin}.\par

In \cite{CNDCDS} we obtained the sharp decay for the derivatives of the solutions in every dimension, namely $\px^{d+m+|\alpha|}\partial^\alpha u \in L^\infty,$ provided $m>0$. In the present case ($m=0$) the situation is more complicate since this case represents a critical threshold for the mapping properties of the convolution operators involved in the sequel, cf.\ the next Proposition \ref{pro2-b}. Nevertheless, the above result gives the optimal decay for the derivatives at least in dimension $d=1$, which includes the equation \eqref{travLCeq}, as it is evident from\ \eqref{solLC}.

In the case when the nonlinear term $F(u)$ is a polynomial, we can also prove the analyticity of the solution and its holomorphic extension in a strip in the complex domain.
We do not include the proof of this result since it can be obtained as in \cite{CNDCDS} without any modification.

\begin{theorem}\label{mainteo3}
Assume \eqref{simbolo}, \eqref{ellitticita2} and let $F$ be a polynomial in $u,\overline{u}$, with $F(0)=0$, $F'(0)=0$.
Let $u$ be a distribution solution of $p(D)u=F(u)$, satisfying $\langle x\rangle ^{\epsilon_0} u\in L^\infty(\R^d)$ for some $\epsilon_0>0$.\par
 Then there exists $\epsilon>0$ such that $u$ extends to a bounded holomorphic function $u(x+iy)$ in the strip $\{z=x+iy\in\mathbb{C}^d:\,|y|<\epsilon\}$.  
\end{theorem}


\section{Notation and preliminary results}
In the sequel the Fourier transform of a function or temperate distribution $f$ is normalized as 
\[
\widehat{f}(\xi)=\Fur f(\xi)=\int_{\rd} e^{-ix\xi} f(x)\, dx.
\]
We already have used in the Introduction the notation $p(D)$ for the Fourier multiplier
\[
p(D) f=(2\pi)^{-d}\int_{\rd} e^{ix\xi} p(\xi)\widehat{f}(\xi)\, d\xi
\]
with symbol $p(\xi)$.
Such an operator is a convolution operator 
\[
p(D)f=K\ast f
\]
with the convolution kernel $K=\Fur^{-1}(p)$. 

Let $w(x)>0$ be a measurable function. We define the weighted Lebesgue spaces
\[
L^\infty_{w}=\{u\in L^\infty(\rd):\ \|u\|_{L^\infty_{w}}:=\|w u\|_{L^\infty}<\infty\}.
\]
In particular we will deal with the weight functions
\[
v_r(x)=\langle x\rangle^r,\ x\in\rd,\quad r\geq 0.
\]


We now recall the following boundedness results on weighted $L^\infty$ spaces (see \cite[Proposition 5]{CNDCDS} and its proof).
\begin{proposition}\label{pro1}
Let $p(D)$ be a Fourier multiplier as in \eqref{simbolo}, \eqref{ellitticita2}. Suppose, in addition that $p_0(\xi)=p_0$ is constant.\par Suppose that there exists $j>0$ such that $p_{m_j}$ is not smooth and let $m$ be the minimum value of such $m_j$ (in particular $m>0$). Then 
\[
p(D)^{-1}:L^\infty_{v_r}\to L^\infty_{v_r}
\]
 continuously for every $0\leq r\leq m+d$.\par If $p_{m_j}$ is smooth for every $j=1,\ldots,h,$ the above conclusion holds for every $r \geq 0$.
\end{proposition}
\begin{proposition}\label{pro1-b}
Let $q(\xi)$ be a symbol in the class $S^{-M}(\rd)$, $M>0$, i.e.\ satisfying the estimates 
\[
|\partial^\alpha q(\xi)|\leq C_\alpha \langle \xi\rangle^{-M-|\alpha|},\quad \alpha\in\N^d,\ \xi\in\rd
\]
for some constants $C_\alpha>0$. Then 
\[
q(D):L^\infty_{v_r}\to L^\infty_{v_r}\]
 for every $0\leq r\leq d+M$.
\end{proposition}
We also need the following classical result about the Fourier transform of homogeneous distributions (for a proof see e.g. \cite[Proposition 4]{CNDCDS}).

\begin{proposition}\label{pro4}
Let $f\in C^\infty(\rd\setminus\{0\})$ be positively homogeneous of degree $r\geq 0$ and $\chi\in C^\infty_0(\rd)$. There exists a constant $C>0$ such that
\[
|\widehat{\chi f}(\xi)|\leq C(1+|\xi|)^{-d-r},\quad \xi\in\rd.
\]
 \end{proposition}
Actually the result in \cite[Proposition 4]{CNDCDS})
was stated for positively homogeneous functions of degree $r>0$ but exactly the same proof holds for $r=0$.\par
The following result was proved in \cite[Proposition 3]{CNDCDS}.
\begin{proposition}\label{pro2-bzero}
Let $Af=K\ast f$ be a convolution operator with integral kernel $K\in L^\infty_{v_s}$, with $s>d$. Then $A$ is bounded on $L^\infty_{v_r}$ for every $0\leq r\leq s$.
\end{proposition}
In the limiting case $s=d$ we have the following result, which is also a key estimate in the sequel.
\begin{proposition}\label{pro2-b}
Let $r>0 $. Let $Af=K\ast f$ be a convolution operator with integral kernel $K\in L^\infty_{v_d}$. Then 
\[
A:L^\infty_{v_r}\to L^\infty_{w_r}
\]
continuously, with
\begin{equation}\label{wr}
w_r(x)=\min\{\langle x\rangle^r/\log(1+\langle x\rangle), \langle x\rangle^d\}\asymp \begin{cases}
\langle x\rangle^r/\log(1+\langle x\rangle)\ &{\rm if}\ 0<r\leq d,\\
\langle x\rangle^d \ &{\rm if}\ r>d.
\end{cases}
\end{equation}
\end{proposition}
\begin{proof}
It is sufficient to prove that, for $r>0$ there exists a constant $C>0$ such that
\[
\int_{\rd} \frac{1}{\langle x-y\rangle^d}\frac{1}{\langle y\rangle^r}\, dy\leq C\max\Big\{\frac{\log(1+\langle x\rangle)}{\langle x\rangle^r}, \frac{1}{\langle x\rangle^d}\Big\},\quad x\in\rd.
\]
We split the integral in the left-hand side in the three regions $|y|\leq |x|/2$, $|x|/2\leq |y|\leq 2|x|$ and $|y|\geq 2|x|$.\par
 When $|y|\leq |x|/2$ we have $|x-y|\asymp |x|$, and therefore 
\begin{align*}
\int_{|y|\leq|x|/2} \frac{1}{\langle x-y\rangle^d}\frac{1}{\langle y\rangle^r}\, dy&\lesssim\frac{1}{\langle x\rangle^d}\int_{|y|\leq|x|/2}\frac{1}{\langle y\rangle^r}\, dy.
\end{align*}
Now, 
\[
\int_{|y|\leq|x|/2}\frac{1}{\langle y\rangle^r}\, dy\lesssim
\begin{cases}
\max\Big\{1,\frac{1}{\langle x\rangle^{r-d}}\Big\}\ &{\rm if}\ r\not=d\\
\log(1+\langle x\rangle) \ &{\rm if}\ r=d.
\end{cases}
\]
Hence we obtain
\[
\int_{|y|\leq|x|/2} \frac{1}{\langle x-y\rangle^d}\frac{1}{\langle y\rangle^r}\, dy\lesssim
\begin{cases}
\max\Big\{\frac{1}{\langle x\rangle^d},\frac{1}{\langle x\rangle^{r}}\Big\}\ &{\rm if}\ r\not=d\\
\log(1+\langle x\rangle)/\langle x\rangle^d \ &{\rm if}\ r=d.
\end{cases}
\]

 When $|x|/2\leq |y|\leq 2|x|$ we have $|y|\asymp |x|$, and therefore
 \begin{align*}
 \int_{|x|/2\leq |y|\leq 2|x|} \frac{1}{\langle x-y\rangle^d}\frac{1}{\langle y\rangle^r}\, dy&\lesssim
 \frac{1}{\langle x\rangle^r}\int_{|x|/2\leq |y|\leq 2|x|}\frac{1}{\langle x-y\rangle^d}\, dy\\
 &\leq \frac{1}{\langle x\rangle^r} \int_{|y-x|\leq 3|x|}\frac{1}{\langle x-y\rangle^d}\, dy\\
 &= \frac{1}{\langle x\rangle^r} \int_{|y|\leq 3|x|}\frac{1}{\langle y\rangle^d}\, dy\lesssim\frac{\log(1+\langle x\rangle)}{\langle x\rangle^r}.
 \end{align*}
 Finally for $|y|\geq 2|x|$ we have $|x-y|\asymp |y|$ and using the assumption $r>0$ we obtain
 \[
\int_{|y|\geq 2|x|}\frac{1}{\langle x-y\rangle^d} \frac{1}{\langle y\rangle^r}\, dy\lesssim \int_{|y|\geq 2|x|} \frac{1}{\langle y\rangle^{r+d}}\, dy\lesssim\frac{1}{\langle x\rangle^r}.
\]
\end{proof}

The following result shows that similar boundedness estimates hold for $p(D)^{-1}$, if $p(D)$ satisfies \eqref{simbolo} and \eqref{ellitticita2}.

\begin{proposition}\label{pro5-b}
Let $p(D)$ be a Fourier multiplier satisfying \eqref{simbolo}, \eqref{ellitticita2}, and $r>0$. Then
\[
p(D)^{-1}=p^{-1}(D): L^\infty_{v_r}\to L^\infty_{w_r}
\]
continuously, with $w_r$ as in \eqref{wr}.
\end{proposition}
\begin{proof}
We consider separately the low and high frequency components of $p(\xi)$. Namely, let $\chi\in C^\infty_0(\R^d)$, $\chi=1$ in a neighborhood of $0$. We write 
\[
p(\xi)^{-1}=\underbrace{(1-\chi(\xi))p(\xi)^{-1}}_{q_1(\xi)}+\underbrace{\chi(\xi)p(\xi)^{-1}}_{q_2(\xi)}.
\]
By \eqref {ellitticita2} we have $q_1\in S^{-M}$ and therefore by Proposition \ref{pro1-b}  we have 
\[
q_1(D): L^\infty_{v_r}\to L^\infty_{v_r}\quad {\rm for}\ 0<r\leq d+M,
\]
as well as
\[
q_1(D): L^\infty_{v_r}\hookrightarrow L^\infty_{v_{d+M}}\to L^\infty_{v_{d+M}} \hookrightarrow L^\infty_{v_d}=L^\infty_{w_r}\quad {\rm for}\  r>d+M
\]
(recall that, for $r>d$, $w_r(x)\asymp \langle x\rangle^{d}=v_d(x)$).
In both cases we obtain $q_1(D): L^\infty_{v_r}\to L^\infty_{w_r}$.\par
Concerning $q_2(D)$ we write 
\[
q_2(\xi)=\frac{\chi(\xi)}{p(\xi)}=\frac{\chi(\xi)}{p_0(\xi)}\cdot \frac{1}{1+\sum_{j=1}^h p_{m_j}(\xi)/p_0(\xi)}
\]
(as observed in Introduction, \eqref{ellitticita2} implies $p_0(\xi)\not=0$ for $\xi\not=0$).\par
By Proposition \ref{pro1} there exists $m>0$ such that the second factor in the right-hand side gives rise to a bounded operator $L^\infty_{v_r}\to L^\infty_{v_r}$ if $0<r\leq d+m$. On the other hand by Proposition \ref{pro4} the first factor $\chi(\xi)/p_0(\xi)$ has (inverse) Fourier transform $K\in L^\infty_{v_d}$ and therefore the corresponding convolution operator is bounded $L^\infty_{v_r}\to L^\infty_{w_r}$ by Proposition \ref{pro2-b}. This gives 
\[
q_2(D): L^\infty_{v_r}\to L^\infty_{w_r}\quad {\rm for}\ 0<r\leq d+m.
\]
The case $r>m+d$ follows from the previous case, using the inclusions of the weighted $L^\infty$ spaces:
\[
q_2(D): L^\infty_{v_r}\hookrightarrow L^\infty_{v_{d+m}}\to L^\infty_{w_{d+m}} = L^\infty_{w_r}\quad {\rm for}\  r>d+m.
\]

\end{proof}
\begin{remark}
Observe that Proposition \ref{pro5-b} does not extend to symbols $p(D)$ of order $0$, namely for $p(D)=p_0(D)$. Consider for example, in dimension $1$, the symbol $p(\xi)={\rm sign}\,\xi$. The corresponding operator $p(D)$ is the Hilbert transform, that is the convolution with the principal value distribution ${\rm P.V.}\, 1/x$ (up to a multiplicative constant). Now, if this operator were bounded $L^\infty_{v_r}\to L^\infty$ for some $r>0$, then for every compact $K\subset\rd$ we would have the estimate
\[
\|{\rm P.V.\,} \frac{1}{x}\ast f\|_{L^\infty}\leq C_K\|f\|_{L^\infty}
\] 
for all test functions $f\in C^\infty_0(K)$. This implies that the distribution ${\rm P.V.}\, 1/x$ has order $0$, and this is not the case.\par
Hence, it is essential for $p(D)$ in \eqref{simbolo} to have order $M>0$.
\end{remark}

We end this section with the the following well-known interpolation inequality, which is proved e.g. in \cite[Proposition 1]{CNDCDS}.
\begin{proposition}\label{interp}
Given $0\leq \ell\leq n$, there exists a constant $C>0$ such that the following inequality holds.
Let $I=I_1\times\ldots\times I_d\subset\rd$, where each $I_j$, $j=1,\ldots,d$, is an interval of the form $[a,+\infty)$ or $(-\infty,a]$. Then
\begin{equation}\label{inter}
\|D^\ell u\|_{L^\infty(I)}\leq C\|u\|_{L^\infty(I)}^{1-\ell/n}\|D^n u\|_{L^\infty(I)}^{\ell/n},
\end{equation}
where we set $\|D^k u\|_{L^\infty(I)}:=\sup_{|\alpha|=k}\|\partial^\alpha u\|_{L^\infty(I)}$, $k\in\mathbb{N}$. \par
In particular, if $u\in L^\infty_{v_s}$ and $D^n u\in L^\infty_{v_r}$ then $D^l u\in L^\infty_{v_\nu}$ with $\nu=(1-\ell/n)s+(\ell/n)r$.

\end{proposition}

\section{Proofs of the results}

\begin{proof}[Proof of Theorem \ref{mainteo}]
We write the equation in integral form as
\[
u=p(D)^{-1}(F(u)).
\]
Since the solution $u$ is by assumption in $L^\infty_{v_{\varepsilon_o}}$, $\epsilon_0>0$,  we have $F(u)\in L^\infty_{v_{p\varepsilon_o}}$.
Now, we can apply Proposition \ref{pro5-b} with $r=p\epsilon_0$ and we obtain $u\in L^\infty_{w_r}$. If $r>d$ then $w_r(x)=\langle x\rangle^d$ and the desired result is proved. If instead $r\leq d$ we have $u\in L^\infty_{w_r}\subset L^\infty_{v_s}$, with $s=\epsilon_0+(p-1)\epsilon_0/2$ (we used $\log(1+\langle x\rangle)\lesssim \langle x\rangle^{(p-1)\epsilon_0/2}$).\par By applying this argument repeatedly, after a finite number of steps we arrive to $u\in L^\infty_{w_r}$, for some $r>d$, which concludes the proof.
\end{proof}

\begin{proof}[Proof of Theorem \ref{mainteo2}]
The proof of \eqref{stime2} is similar to that of \cite[Theorem 1.2]{CNDCDS}, but we provide the details for the benefit of the reader.\par
 First of all we show that $u \in H^s$ for any $s \in \R$. Observe that we can write
$$u=p(D)^{-1}(F(u)).$$ By Theorem \ref{mainteo}, we have $u \in L^\infty \cap L^2,$ which gives $F(u) \in L^2$. Moreover the operator $p(D)^{-1}$ has symbol $p(\xi)^{-1}$, hence by \eqref{ellitticita2} it maps continuously $H^s$ into $H^{s+M}$ for every $s \in \R$. This gives $ u \in H^M \cap L^\infty$. Iterating this argument we obtain $u \in H^s$ for every $ s \in \R$. \par To prove \eqref{stime2} we argue by induction on $|\alpha|$. For $|\alpha|=0$ we have in fact the stronger information $u\in L^\infty_{v_d}$ by Theorem \ref{mainteo}. Assume now\eqref{stime2} to be true for $|\alpha| \leq N-1$ and let us prove it for $|\alpha|=N$. First of all, by Proposition \ref{interp} with $l =1, n$ large enough and $u$ replaced by $D^{N-1}u$ we have that 
\begin{equation}\label{dnu}
D^N u \in L^\infty_{v_{d+N-1-\varepsilon}}, \qquad |\alpha|=N
\end{equation} for every $\varepsilon >0$. Let $\chi \in C_0^\infty(\R^d), \chi =1$ around the origin. By introducing commutators in the equation \eqref{equation} we have, for $|\beta| \leq |\alpha| =N$:
\begin{multline}\label{stima5}
x^\beta\partial^\alpha u
=\underbrace{p(D)^{-1}[(\chi p)(D),x^\beta]\partial^\alpha u}_{=:q_1(D)u}
+\underbrace{p(D)^{-1}[((1-\chi) p)(D),x^\beta]\partial^{\alpha} u}_{=:q_2(D)u}\\
+p(D)^{-1}(x^\beta \partial^\alpha F(u)).
\end{multline}
To obtain \eqref{stime2} it is sufficient to prove that the three terms in the right-hand side of \eqref{stima5} are in $L^\infty_{v_{d-\varepsilon}}$ for every $\varepsilon >0$. We first consider $q_1(D)u$. By direct computation we can write
\[
[(\chi p)(D),x^\beta]\partial^\alpha u=-\sum_{0\not=\gamma\leq\beta}i^{|\gamma|}\binom{\beta}{\gamma}\big(\partial^\gamma_\xi (\chi p)\big)(D)(x^{\beta-\gamma}\partial^\alpha u),
\]
where the derivatives of $\chi p$ in the right-hand side are meant in the sense of distributions.
By the inverse Leibniz formula\footnote{Namely, 
\[
x^\beta\partial^\alpha
u(x)=\sum_{\gamma\leq\beta,\,\gamma\leq\alpha}\frac{(-1)^{|\gamma|}\beta!}{(\beta-\gamma)!}
\binom{\alpha}{\gamma}\partial^{\alpha-\gamma}(x^{\beta-\gamma}u(x)).
\]
} we obtain (since $|\beta|\leq|\alpha|$)
\begin{equation}\label{stima6}
[(\chi p)(D),x^\beta]\partial^\alpha u=\sum_{0\not=\gamma\leq\beta}\sum_{\tilde{\alpha},\tilde{\beta}:|\tilde{\beta}|\leq|\tilde{\alpha}|<|\alpha|}  
C_{\alpha,\beta,\tilde{\alpha},\tilde{\beta}}
(\partial^{\ga}_\xi (\chi p))(D)\partial^{\tilde{\gamma}} (x^{\tilde{\beta}}\partial^{\tilde{\alpha}}u).
\end{equation}
where $\tilde{\gamma}$ is a suitable multi-index depending on $\alpha,\beta,\tilde{\alpha},\tilde{\beta},\gamma$, with $|\tilde{\gamma}|=|\gamma|$, and $C_{\alpha,\beta,\tilde{\alpha},\tilde{\beta}}$ are suitable constants. Moreover we have $x^{\tilde{\beta}}\partial^{\tilde{\alpha}}u\in L^\infty_{v_{d-\epsilon}}$ by the inductive hypothesis. Now we observe that the operator $\partial^{\ga}_\xi (\chi p)(D)\partial^{\tilde{\gamma}}$ has a symbol given by a sum of homogeneous functions of nonnegative degree multiplied by cut-off functions, hence by Proposition \ref{pro4} we deduce that its integral kernel belongs to $L_{v_d}^\infty$, so that it maps $L_{v_{d-\epsilon}}^\infty\to L_{w_{d-\epsilon}}^\infty\hookrightarrow L^\infty_{v_{d-2\varepsilon}}$ by Proposition \ref{pro2-b}. The same happens for $p(D)^{-1}:L^\infty_{v_{d-2\varepsilon}}\to L^\infty_{v_{d-3\varepsilon}}$, and therefore we have $q_1(D):L_{v_{d-\epsilon}}^\infty \to L^\infty_{v_{d-3\varepsilon}}$ for every $\varepsilon >0$, and therefore $q_1(D)u\in L^\infty_{v_{d-\epsilon}}$ for every $\epsilon>0$. \par 
Concerning $q_2(D)u$, by the symbolic calculus we have 
\begin{align}
q_2(D)u&= p(D)^{-1}[((1-\chi) p)(D),x^\beta]\partial^\alpha u\nonumber\\
&=-\sum_{0\not=\gamma\leq\beta}i^{|\gamma|}\binom{\beta}{\gamma}p(D)^{-1} \big(\partial^\gamma_\xi ((1-\chi)p)\big)(D)(x^{\beta-\gamma}\partial^\alpha u).\label{qpq}
\end{align}
Using \eqref{dnu} and $|\beta-\gamma|\leq N-1$ we see that $x^{\beta-\gamma}\partial^\alpha u\in L^\infty_{d-\epsilon}$. On the other hand, the multiplier $p(D)^{-1}(\partial^{\ga}_\xi ((1-\chi) p))(D)$ has a smooth symbol of negative order and therefore by Proposition \ref{pro1-b} it maps $L_{v_{r}}^\infty$ into itself continuously if $0<r\leq d$.  Hence $q_2(D) u\in L^\infty_{d-\epsilon}$. 

\par
Finally, concerning the nonlinear term, by the Fa\`a di Bruno formula and the interpolation inequalities we have 
\begin{equation}\label{eqq}
\|D^N F(u)\|_{L^\infty(I)}\leq C\sum_{1\leq \nu\leq N} \|F'\|_{C^{\nu-1}}\|u\|_{L^\infty(I)}^{\nu-1}\|D^N u\|_{L^\infty(I)}.
\end{equation}
cf. \cite[Formula (3.1.9)]{taylor}. In this formula, we use the same notation as in Proposition \ref{interp}, and the norm of $F'$ is meant on the range of $u$.\par We now estimate each term by observing that, for $\nu>1$ (integer)  we have $|u|^{\nu-1}\in L^\infty_{(\nu-1)d}$, which implies $\langle x\rangle |u|^{\nu-1}\in L^\infty$; when $\nu=1$ instead $|F'(u)|\leq C|u|\in L^\infty_{v_d}$, so that $\langle x\rangle F'(u)\in L^\infty$. On the other hand, we also know from \eqref{dnu} that $\partial^\alpha u\in L^\infty_{v_{d+N-1-\epsilon}}$ if $|\alpha|=N$. We deduce by \eqref{eqq} that $D^N F(u)\in L^\infty_{v_{d+N-\varepsilon}}$. By Proposition \ref{pro5-b} we get $p(D)^{-1}(x^\beta\partial^\alpha F(u))\in L^\infty_{w_{d-\varepsilon}} \subset L^\infty_{v_{d-2\varepsilon}} $ for every $\varepsilon >0$. The estimate \eqref{stime2} is then proved.\par
Let us now prove \eqref{stime2bisbis} in the case $d=1$. We argue as in the first part of the present proof, assuming \eqref{stime2bisbis} for $|\alpha|\leq N-1$ and prove it for $|\alpha| \leq N$. Moreover by the first part of the present proof we already know that 
\begin{equation}\label{pfpp}
D^N u\in L^\infty_{v_{1+N-\varepsilon}}.
\end{equation}
We will exploit the particularly simple structure of positively homogeneous functions in dimension $d=1$, which allows us to eliminate the above $\epsilon$-loss in the decay.\par 
Let us estimate $q_1(D)u$ in \eqref{stima5}. We observe that in \eqref{stima6} we have $ x^{\tilde{\beta}}\partial^{\tilde{\alpha}}u\in L^\infty_{v_1}$ by the inductive hypothesis. Moreover we can split further the right-hand side of \eqref{stima6}: by the Leibniz formula and using $p(\xi)=\sum_{j=0}^h p_j(\xi)$ we see that the operator $p(D)^{-1}\partial^{\ga}_\xi (\chi p)(D)\partial^{\tilde{\gamma}}$ (now with $\tilde{\gamma}=\gamma$ because $|\tilde{\gamma}|=|\gamma|$ and $d=1$) has as symbol $i^\gamma$ times the function
\[
\sum_{0\leq\mu<\gamma}\binom{\gamma}{\mu}p(\xi)^{-1}\partial^{\gamma-\mu}\chi(\xi)\xi^\gamma\partial^\mu p(\xi)+\sum_{j=1}^h\chi(\xi)p(\xi)^{-1}\xi^\gamma\partial^\gamma p_j(\xi)+\chi(\xi)p(\xi)^{-1}\xi^\gamma\partial^\gamma p_0(\xi).
\]
The first sum is a smooth compactly supported symbol and therefore by Proposition \ref{pro1-b} gives rise to a bounded operator $L^\infty_{v_1}\to L^\infty_{v_1}$.\par
 For the second term we write
\[
\chi(\xi)p(\xi)^{-1}\xi^\gamma\partial^\gamma p_j(\xi)=\frac{\chi(\xi)\xi^\gamma\partial^\gamma p_j(\xi)}{p_0(\xi)}\cdot \frac{1}{1+\sum_{j=1}^h p_{m_j}(\xi)/p_0(\xi)}
\]
which therefore gives rise to a bounded operator $L^\infty_{v_1}\to L^\infty_{v_1}$ by Propositions \ref{pro4}, \ref{pro2-bzero} and \ref{pro1}, since $\xi^\gamma\partial^\gamma p_j(\xi)/p_0(\xi)$ has degree $m_j>0$.\par
 Finally the last term \[
\chi(\xi)p(\xi)^{-1} \xi^\gamma\partial^\gamma p_0(\xi)\]
 vanishes identically, because in dimension $1$ we have $p_0(\xi)=p_0^+$ if $\xi>0$ and $p_0(\xi)=p_0^-$ for $\xi<0$ ($p_0^+,p_0^-\in\C$), so that $\partial^\gamma p_0(\xi)=(p_0^+ -p_0^-)\partial^{\gamma-1}\delta$ and $\xi^\gamma \partial^{\gamma-1}\delta=0$. This is in fact the only point where the assumption $d=1$ plays an essential role.
\par
Concerning the term $q_2(D)u$ in \eqref{qpq}, we can use the information \eqref{pfpp} which implies $x^{\beta-\gamma}\partial^\alpha u\in L^\infty_{v_{2-\epsilon}}\subset L^\infty_{v_{1}}$ (if $\epsilon<1$) because $|\beta-\gamma|\leq N-1$, and we conclude that $q_2(D)u\in L^\infty_{v_1}$, since $p(D)^{-1}(\partial^{\ga}_\xi ((1-\chi) p))(D)$ is bounded on $L^\infty_{v_1}$, as already observed.\par
Finally, for the nonlinear term $p(D)^{-1}(x^\beta \partial^\alpha F(u))$ in \eqref{stima5} we can argue as above, but now we use again the information \eqref{pfpp} in place of \eqref{dnu} (with $d=1$). By \eqref{eqq} we now obtain $D^N F(u)\in L^\infty_{v_{2+N-\varepsilon}}$. By Proposition \ref{pro5-b} we have $p(D)^{-1}:L^\infty_{v_{2-\epsilon}}\to L^\infty_{w_{2-\epsilon}}=L^\infty_{v_1}$ (if $\epsilon<1$), so that $p(D)^{-1}(x^\beta\partial^\alpha F(u))\in L^\infty_{v_1}$. This concludes the proof of \eqref{stime2bisbis}.

\end{proof}

\end{document}